\documentclass[12pt,a4paper]{amsart}
\usepackage[T2A]{fontenc}
\usepackage{amsmath}
\usepackage{amssymb,amscd}
\usepackage{enumerate}

\usepackage{fullpage}
\usepackage{hyperref}
\usepackage{ulem}
\theoremstyle{plain}
\newtheorem{theorem}{Theorem}[section]
\newtheorem{lemma}[theorem]{Lemma}
\newtheorem{proposition}[theorem]{Proposition}
\newtheorem{corollary}[theorem]{Corollary}
\newtheorem{question}[theorem]{Question}

\theoremstyle{definition}
\newtheorem{definition}[theorem]{Definition}
\newtheorem{example}[theorem]{Example}

\newtheorem{remark}[theorem]{Remark}

\theoremstyle{remark}

\def\Aut{\operatorname{Aut}}
\def\Bir{\operatorname{Bir}}
\def\End{\operatorname{End}}

\def\Der{\operatorname{Der}}
\def\LND{\operatorname{LND}}
\def\Frac{\operatorname{Frac}}
\def\Lie{\operatorname{Lie}}
\def\Spec{\operatorname{Spec}}

\def\ord{\operatorname{ord}}

\def\CC{{\mathbb C}}

\def\AA{{\mathbb A}}
\def\NN{{\mathbb N}}
\def\K{{\mathbb K}}
\def\LL{{\mathbb L}}
\def\FF{{\mathbb F}}
\def\G{{\mathbb G}}

\def\cO{{\mathcal O}}

\def\embed{\hookrightarrow}

\def\J{{\mathcal J}}

\makeatletter
\newcommand\thankssymb[1]{\lowercase{\textsuperscript{\@alph{#1}}}}
\makeatother

\address{
HSE University, Faculty of Computer Science\\
Pokrovsky blvd. 11, Moscow, 109028
Russia
}

\email{a@perep.ru}

 \usepackage{xcolor}

\DeclareMathOperator{\id}{id}

\DeclareMathOperator{\Jonq}{Jonq}
\DeclareMathOperator{\jonq}{\mathfrak{jonq}}

\begin{document}

\title[Nested automorphism subgroups]{Structure of connected nested automorphism groups}

\author{Alexander Perepechko}
\thanks{
   The research was supported by the grant RSF-DST 22-41-02019.}

\begin{abstract} 
   In this article, we describe the maximal unipotent subgroups of $\mathrm{Aut}(X)$, where $X$ is an affine algebraic variety. Every subgroup of this type has a structure analogous to that of the group of triangular automorphisms of $\mathbb{A}^n$. In particular, it is \emph{nested}, that is, a countable increasing union of algebraic subgroups.

We show that a subgroup $G\subset\mathrm{Aut}(X)$ consisting of unipotent elements is closed if and only if it is nested. This implies that a connected nested subgroup of $\mathrm{Aut}(X)$ is closed, thus answering a question posed by Kraft and Zaidenberg (2022).

We also extend the recent description of maximal commutative unipotent subgroups of $\mathrm{Aut}(X)$ due to Regeta and van Santen (2024), by providing a direct construction of such subgroups within our approach.
\end{abstract}

\maketitle

\section{Introduction}

The cornerstone of the structure theory of algebraic groups, the Lie--Kolchin triangularization theorem, states that any unipotent subgroup of a matrix group $\mathrm{GL}_n(\CC)$ is conjugate to a subgroup of upper triangular matrices. 
Automorphism groups of algebraic varieties are a natural generalization of linear groups, yet their structure is significantly more complex and far less understood. 
In particular, the notion of triangular automorphisms is defined only for the affine space, and there exist unipotent subgroups that are not triangularizable by conjugation.
We generalize this notion to an arbitrary affine variety so that an analogue of the Lie--Kolchin theorem holds.

Let $X$ be an irreducible affine algebraic variety over an algebraically closed field $\K$ of characteristic zero.
The group $\Aut(X)$ of regular automorphisms of $X$ carries the structure and topology of an ind-group. 
An ind-group is, informally, an infinite-dimensional analogue of an algebraic group; see~\cite{FK} for a precise treatment.
Recall that a subgroup $G \subset \Aut(X)$ is called \emph{nested} if it admits a countable ascending filtration by algebraic subgroups $G_i$.

The notion of unipotent subgroup can be naturally extended to an abstract subgroup of $\Aut(X)$. Namely, we call an element $g\in\Aut(X)$ \emph{unipotent} if it is contained in a subgroup of $\Aut(X)$ isomorphic to the additive group of the field $\G_a:=(\K,+)$; we call a subgroup $G\subset\Aut(X)$ \emph{unipotent} if it consists of unipotent elements.

In the case of $X=\AA^n$, there is a distinguished \emph{de Jonqui\`eres} subgroup of triangular automorphisms of $\AA^n$. 
Its unipotent radical is as follows:
\begin{equation}\label{eq:jonq-intro}
   \Jonq_u(n,\K)=\{(x_1,\ldots,x_n)\mapsto (x_1+P_0, x_2+P_1,\ldots,x_n+P_{n-1})\mid P_i\in\K[x_1,\ldots,x_i]\}.   
\end{equation}
This radical has the following remarkable property. 
\begin{theorem}[{\cite[Theorem D]{KZ}}]\label{th:KZ-intro}
   Let $U \subset \Aut(\AA^n)$ be a nested unipotent subgroup. If~$U$ has a dense orbit on $\AA^n$, then $U$ is conjugate to a subgroup of $\Jonq_u(n,\K)$.
\end{theorem}

We establish a similar property for an arbitrary affine variety $X$ and a nested unipotent subgroup $U$ of $\Aut(X)$.
Let us briefly describe the construction. 

Let $\cO(X)$ denote the algebra of regular functions on $X$, and consider the subalgebra of $U$-invariants $\cO(X)^U$.
We find an element $h\in\cO(X)^U$ such that the localization $\cO(X)_h$ is a polynomial algebra over the subring $R:=\cO(X)_h^U$ in, say, $k$ variables. Then $U$ is contained in the group $\Jonq_u(k,R)\subset\Aut(X_h)$ obtained by substituting $R$ for $\K$ in \eqref{eq:jonq-intro}.

We define the \emph{width} of a subgroup $G\subset\Aut(X)$ as the dimension of a general $G$-orbit in $X$.
Then groups of the form 
\[\Jonq_u(k,\cO(X)_h^U)\cap\Aut(X)\subset\Aut(X)\] 
are exactly maximal subgroups among unipotent subgroups of a given width. 
We call them \emph{de Jonqui\`eres-like} subgroups (or, for short, \emph{dJ-like}).
The construction above provides a de Jonqui\`eres-like subgroup of the same width as the initial unipotent subgroup~$U$ (Proposition~\ref{pr:U-wide-dJ}).

The second major topic of this paper concerns the topology of automorphism subgroups.
In~\cite{KZ}, the authors explore structural and topological properties of various classes of algebraically generated subgroups. 
In particular, they ask in~\cite[Question 4]{KZ} whether a connected nested subgroup of $\Aut(X)$ is closed. 
This is motivated by the fact that algebraic subgroups of $\Aut(X)$ are closed.
The following theorem gives a positive answer to this question.

\begin{theorem}[{Theorem~\ref{th:nested-closed}}] \label{th:nested-closed-intro}
   A connected nested subgroup $G\subset\Aut(X)$ is closed.
\end{theorem}

Moreover, we establish the following criterion for a unipotent subgroup to be closed.

\begin{theorem}[{Theorem~\ref{th:nested=closed}}] \label{th:nested=closed-intro}
  Assume that the base field $\K$ is uncountable and consider a unipotent subgroup $U$ of $\Aut(X)$. Then the following are equivalent:
   \begin{enumerate}[(i)]
      \item $U$ is generated by its $\G_a$-subgroups;
      \item $U$ is nested;
      \item $U$ is closed in $\Aut(X)$.
   \end{enumerate}
\end{theorem}

We now describe the contents of the paper in more detail.
A central role is played by Lie subalgebras of derivations of $\cO(X)$ contained in the subset $\LND(\cO(X))$ of locally nilpotent derivations; these serve as the infinitesimal counterpart of nested unipotent subgroups.
We also introduce the notion of \emph{generic frames} --- tuples of pairwise commuting locally nilpotent derivations whose vector fields are linearly independent at a general point --- which play an equally central role.

In Section~\ref{sec:nested}, we show that any Lie subalgebra $L\subset\LND(\cO(X))$ embeds, after localization, into a de Jonqui\`eres subalgebra (Lemma~\ref{lm:L-structure}).
This leads to a bijective correspondence between Lie subalgebras in $\LND(\cO(X))$ and nested unipotent subgroups of $\Aut(X)$ via the exponential map (Proposition~\ref{pr:exp-nested}).
We extend this correspondence to an isomorphism of ind-varieties (Corollary~\ref{cr:jonq-closed}) and prove that a nested unipotent subgroup is closed both in the semigroup of regular endomorphisms $\End(X)$ and in $\Aut(X)$ (Proposition~\ref{pr:unipotent-closed}).
From this we deduce Theorem~\ref{th:nested-closed-intro}.
 In Section~\ref{sec:non-nested}, we prove Theorem~\ref{th:nested=closed-intro}.

In Section~\ref{sec:dJ-like}, we describe dJ-like subgroups in terms of generic frames (Proposition~\ref{pr:dJ-like=Jdi}), and we enumerate all dJ-like subgroups contained in a given one (Proposition~\ref{pr:J'<J}).
In Section~\ref{sec:max-ex}, we utilize this description to provide examples of maximal nested unipotent subgroups in dimensions 2 and 3.

In Appendix~\ref{sec:max-commut}, we utilize generic frames to describe maximal commutative unipotent subgroups (Proposition~\ref{pr:max-commut}), extending the recent description by Regeta and van Santen~\cite{RvS}, and to establish when a maximal commutative subgroup is contained in a given dJ-like subgroup (Proposition~\ref{pr:R<J}). We also deduce a maximality criterion for dJ-like subgroups (Proposition~\ref{pr:dJ-max-criterion}).

The author is grateful to Ivan Arzhantsev, Nikhilesh Dasgupta, Sergey Gaifullin, Vsevolod Gubarev, 
Neena Gupta, and Andriy Regeta for numerous useful discussions and remarks.
The author thanks Mikhail Zaidenberg for his lasting encouragement and motivation.
Finally, the author thanks the referees for their thorough reports that substantially shaped the final form of the paper.

\section{Preliminaries}\label{sec:prelim}

\subsection{Ind-groups}\label{sec:ind-groups}

The notion of an ind-group goes back to Shafarevich~\cite{Sh66}. 
We refer to~\cite{FK} and~\cite{KZ24} for a thorough introduction.

\begin{definition}
   An ind-variety $V$ is a set endowed with an ascending filtration
   $V_0 \embed V_1 \embed V_2 \embed \ldots \subset V$ such that:
   \begin{enumerate}[(i)]
   \item $V = \bigcup_{k \in \NN} V_k$;
   \item each $V_k$ is an algebraic variety;
   \item for all $k \in \NN$ the embedding $V_k \embed V_{k+1}$ is a closed immersion.
   \end{enumerate}
\end{definition}
   
   An ind-variety $V$ is called \emph{affine} if all $V_i$ are affine.
An ind-variety $V$ has a natural \emph{topology}: a subset $S \subset V$ is called open (resp. closed, locally closed) if $S_k := S \cap V_k \subset V_k$ is open (resp. closed, locally closed) for all $k \in \mathbb{N}$.
A closed subset $S \subset V$ has a natural structure of an ind-variety and is called an ind-subvariety.

A \emph{morphism} between ind-varieties $V = \bigcup_k V_k$ and $W = \bigcup_m W_m$ is a map $\phi: V \to W$ satisfying the following condition. For every $k \in \mathbb{N}$ there is an $m \in \mathbb{N}$ such that 
$\phi(V_k) \subset W_m$ and that the induced map $V_k \to W_m$ is a morphism of algebraic varieties. The product of ind-varieties $X=\bigcup_i X_i$ and $Y=\bigcup_i Y_i$ is defined as $\bigcup_i (X_i\times Y_i)$. 
 Recall the following definition.

\begin{definition}
   An \emph{ind-group} is an ind-variety with a group structure such that the multiplication and inverse maps are morphisms of ind-varieties. 
\end{definition}

If $H$ is a closed subgroup of an ind-group $G$, then $H$ is again an ind-group with respect to the induced ind-variety structure.
 
For an affine variety $X$ we use the ind-group structure on $\Aut(X)=\bigcup\Aut(X)_{\le d}$, where 
\[
   \Aut(X)_{\le d}=\{\psi\mid\deg\psi,\deg\psi^{-1}\le d\},
\]
where the degree is induced by an embedding $X\embed\AA^N$, e.g., see~\cite[Section~5.2]{FK}.

This structure has the universal property that a map $\rho\colon H\to\Aut(X)$ from an algebraic variety $H$ is a morphism if and only if the corresponding action map $H\times X\to X$ is a morphism, see~\cite[Theorem~5.1.1]{FK}.
Then $\rho$ is called an \emph{algebraic family of automorphisms} of $X$, see also~\cite{Ram}.
Moreover, if $H$ is an algebraic group, then the image $\rho(H)$ is closed.
Indeed, its closure is an algebraic subgroup that contains $\rho(H)$ as a dense constructible subset. 

Two ind-structures $V=\bigcup_i V_i$ and $V=\bigcup_i V_i^\prime$
are called \emph{equivalent} if the identity map $\bigcup_i V_i\to \bigcup_i V_i^\prime$ is an isomorphism of ind-varieties.
 An ind-group $G$ is called \emph{nested} if it admits an equivalent
filtration $G=\bigcup_i G_i$, where all $G_i$ are algebraic subgroups.

\subsection{Ind-structures for a ring}\label{sec:ind-ring}
An affine domain $R$ over $\K$ is endowed with a natural structure of an ind-variety via choosing an ascending sequence of finite-dimensional subspaces $V_1\subset V_2\subset\ldots\subset R$ such that $R=\bigcup_i V_i$. Note that for any $i,j>0$ the product $V_i V_j$ belongs to $V_k$ for some $k$.
In general, for any $\K$-space of countable dimension --- whether a subalgebra of regular functions or a Lie subalgebra of derivations --- we introduce an ind-structure on it by considering an arbitrary filtration by finite-dimensional subspaces.
Different filtrations give rise to equivalent ind-structures. 

The monoid of $\K$-endomorphisms $\End(R):=\End_\K(R)$ carries an induced ind-structure.
Indeed, let $r_1,\ldots,r_s\in R$ be a finite set of generators of $R$ as a $\K$-algebra. Then we have an embedding 
\[\End(R)\embed R^s,\, \phi\mapsto (\phi(r_1),\ldots,\phi(r_s)).\]
Its image is a closed ind-subvariety, see~\cite[Proposition~13.1.2]{FK}.

\subsection{Lie algebras of ind-groups}
For an ind-variety $V = \bigcup_{k \in \mathbb{N}} V_k$ we define the Zariski tangent space at $x \in V$ as follows. We have $x \in V_k$ for $k \ge k_0$, and
$T_x V_k \subset T_x V_{k+1}$ for $k \ge k_0$, and then we define
\[
  T_x V := \bigcup_{k \ge k_0} T_x V_k.
\]
Thus, $T_x V$ is a vector space of at most countable dimension. 

For an ind-group $G$, the tangent space $T_e G$ has a natural structure of a Lie algebra which is denoted by $\Lie G$; see~\cite[Section~4]{Ku} and~\cite[Section~2]{FK}. 
There is a natural embedding of $\Lie\Aut(X)$ into the Lie algebra of vector fields on $X$, see~\cite[Proposition~7.2.4]{FK}, and a canonical identification of the vector fields on $X$ and the ($\K$-linear) derivations of $\cO(X)$, see~\cite[Section~3.2]{FK}.
 
\subsection{Unipotent subgroups}
 An element $u\in\Aut(X)$ is called \emph{unipotent} if $u$ 
 belongs to an algebraic subgroup of $\Aut(X)$ isomorphic to $\mathbb{G}_a=(\K,+)$. A subgroup $G\subset\Aut(X)$ is called \emph{unipotent} if it consists of unipotent elements.

 \subsection{LNDs}
 By a derivation of the algebra of regular functions $\cO(X)$ we mean a $\K$-derivation and denote the algebra of all derivations of $\cO(X)$ by $\Der(\cO(X))$.
We denote by $\ker\partial$ the kernel of a derivation $\partial\in\Der(\cO(X))$ in $\cO(X)$, and by $\ker_{\K(X)}\partial$ the kernel of $\partial$ in the field of rational functions $\K(X)$. Then $\ker\partial\subset\cO(X)$ is an algebraically closed $\K$-subalgebra and $\ker_{\K(X)}\partial\subset\K(X)$ is an algebraically closed subfield, see~\cite[Propositions~1.9, 1.12]{Freudenburg}.
Given a subset $S\subset\Der(\cO(X))$, we also let $\ker S=\bigcap_{\partial\in S}\ker \partial$ and $\ker_{\K(X)} S=\bigcap_{\partial\in S}\ker_{\K(X)} \partial$.

A derivation $\partial\in\Der(\cO(X))$ is called \emph{locally nilpotent} if for every $f\in\cO(X)$ there exists $k\in\NN$ such that $\partial^k(f)=0$.
 We denote the subset of locally nilpotent derivations (LNDs for short) on $\cO(X)$ by $\LND(\cO(X))$.
    Two LNDs $\partial_1,\partial_2\in\Der\cO(X)$ are called \emph{equivalent} if $\ker\partial_1=\ker\partial_2$.
 It is known that if $\partial_1\sim\partial_2$, then $a\partial_1=b\partial_2$ for some $a,b\in\ker\partial_1=\ker\partial_2$. 
 Equivalently, $\partial_2=f\partial_1$ for some $f\in\ker_{\K(X)}\partial_1$. 
The subset $\LND(\cO(X))$ is a Lie subalgebra of $\Der(\cO(X))$ if and only if any two LNDs of $\cO(X)$ are equivalent, see, e.g.,~\cite{PR1}.

     Given a derivation $\partial\in \Der(\cO(X))$, we also consider the sets $\ker\partial\cdot\partial$ of ``replicas'' and $(\ker_{\K(X)}\partial)\cdot\partial$
    of ``rational replicas'' of $\partial$, see~\cite{RvS}.
Given a derivation $\partial\in\Der(\cO(X))$ and a rational function $f\in\ker_{\K(X)}\partial$, the product $f\partial\in \Der(\K(X))$ may not preserve $\cO(X)$.

The exponential map $\exp\colon \LND(\cO(X))\to\Aut(X)$ sends LNDs to unipotent elements, see, e.g.,~\cite[Section~1.1.9]{Freudenburg} and~\cite[Section~11.3]{FK}.
 By~\cite[Lemma 11.3.3]{FK}, $\exp$ is an $\Aut(X)$-equivariant bijective map from $\LND(\cO(X))$ to the set of unipotent elements of $\Aut(X)$. Its inverse is denoted by $\log$, see~\cite[Definition 11.3.4]{FK}.

 \subsection{De Jonqui\`eres subgroup}

      Consider a polynomial ring $R[x_1,\ldots,x_k]$ over a commutative domain $R$. 
      Recall that the \emph{de Jonqui\`eres subgroup} over $R$ is the subgroup
      \[
         \Jonq(k,R):=\{(x_1,\ldots,x_k)\mapsto (c_1x_1+P_1,\ldots,c_kx_k+P_k)\mid c_i\in R^\times, P_i\in R[x_1,\ldots,x_{i-1}]\}
      \]
      of $\Aut_R(R[x_1,\ldots,x_k])$. 
      We also use the notation $\Jonq(x_1,\ldots,x_k,R)$ to avoid ambiguity.
       It is a solvable nested subgroup, see~\cite[Remark~5.2.1]{KZ}.

      The \emph{unipotent radical} $\Jonq_u(k,R)$ consists of elements with $c_1=\ldots=c_k=1$ and has derived length $k$.
      Indeed, the derived series of $\Jonq_u(x_1,\ldots,x_k,R)$ is as follows:
      \[
      \Jonq_u(x_1,\ldots,x_k,R)\supset
      \Jonq_u(x_2,\ldots,x_k,R[x_1])\supset
      \Jonq_u(x_3,\ldots,x_k,R[x_1,x_2])\supset\ldots
      \]
      We denote the tangent algebra $\Lie\Jonq_u(k,R)$ by $\jonq_u(k,R)$. It equals 
      \[
        \jonq_u(k,R)=\bigoplus_{i=1}^k R[x_1,\ldots,x_{i-1}]\frac{\partial}{\partial x_i}.
      \]
 By~Theorem~\ref{th:KZ-intro}, any nested unipotent subgroup of $\Aut(\AA^n)$ that acts transitively on $\AA^n$ is conjugate to a subgroup of $\Jonq_u(n,\K)$.

\section{Nested Groups}\label{sec:nested}

\subsection{Lie subalgebras of LNDs}\label{sec:ln-subsets}
We recall results on locally nilpotent subsets following~\cite{Skutin}, see also expositions in~\cite{DaiLN} and~\cite{BPZ}.
Let $B$ be a unital commutative algebra that is an integral domain of finite transcendence degree over a field $\mathbb{L}$ of characteristic zero.

\begin{definition}[{\cite[Definition 2]{Skutin}}]\label{def:loc-nilp}
    A subset of $\LL$-derivations $T\subset\Der_\mathbb{L}(B)$ is said to be \emph{locally nilpotent} if for every $b\in B$ and every infinite sequence $\overline{d}=(d_1,d_2,\ldots)$ in $T$ there exists $k\in\NN$, depending on $b$ and $\overline{d}$, such that \[(d_k\circ d_{k-1}\circ\cdots\circ d_1)(b)=0.\]
\end{definition}

By~\cite[Proposition]{Skutin}, every Lie subalgebra $A$ of $\Der_\mathbb{L}(B)$ contained in $\LND_\mathbb{L}(B)$ forms a locally nilpotent subset of derivations. 
The following theorem, which is a consequence, is the key ingredient in the proof of Lemma~\ref{lm:L-structure}.

\begin{theorem}[{\cite[Theorem 6]{Skutin}}]\label{th:Skutin}
   Let $L\subset \Der_\mathbb{L}(B)$ be a Lie subalgebra contained in $\LND_\mathbb{L}(B)$ with $\ker L = \mathbb{L}$. Then there exist $x_1,\ldots,x_n\in B$ algebraically independent over~$\mathbb{L}$ such that $B = \mathbb{L}[x_1, \ldots, x_n]$ and 
   \[L \subseteq \bigoplus_{i=1}^n \mathbb{L}[x_1, \ldots , x_{i-1}]\frac{\partial}{\partial x_i}.\]
\end{theorem}

Lie subalgebras of $\Der(\cO(X))$ contained in $\LND(\cO(X))$ can be described as follows.

\begin{lemma}\label{lm:L-structure}
Let $L \subset \Der( \cO(X))$ be a Lie subalgebra contained in $\LND( \cO(X))$,
and let $A:=\ker L\subset\cO(X)$ denote the common kernel of all derivations in $L$. 
Then there exists a non-zero $h \in A$ and elements $x_1,\dots, x_k \in \cO(X)$ such that the following holds:
\begin{enumerate}[(i)]
    \item $A_h$ is finitely generated;
    \item $x_1,\ldots,x_k$ are algebraically independent over $A$ and
          \[  \cO(X)_h = A_h[x_1,\ldots,x_k];\]
    \item the subalgebra $A_h[x_1,\ldots,x_j]$ is stable under~$L$ for any $j=1,\ldots,k$;
    \item $L$ is a Lie subalgebra of $\jonq_u(k,A_h)$.
\end{enumerate}
\end{lemma}

\begin{proof} 
Since the field of fractions $\FF := \Frac(A)$ of $A$ is algebraically closed in $\K(X)$, 
$B := \FF \otimes_A \cO(X)$ is a domain. Then
$L_\FF := \FF \otimes_A L \subset \Der_\FF(B)$ is a Lie subalgebra contained in $\LND_\FF(B)$, and its kernel in $B$ is $\FF$.
By Theorem~\ref{th:Skutin}, there exist algebraically independent over $\FF$ elements $x_1,\dots,x_k$ of $B$ such that
\begin{equation}\label{eq:LF}
    B = \FF[x_1,\dots,x_k]
    \qquad\text{and}\qquad
    L_\FF \subseteq \bigoplus_{i=1}^k  \FF[x_1,\dots,x_{i-1}]\frac{\partial}{\partial x_i}.   
\end{equation}

Write a finite generating set of $\cO(X)$ as polynomials in $x_1,\dots,x_k$ over $\FF$, and let $g \in A$ be a common denominator of their coefficients.
By~\cite[Proposition~1.1]{OnYo82}, there exists a multiple $h\in (g) \subset A$ such that $A_h$ is finitely generated.
Then $ \cO(X)_h = A_h[x_1,\dots,x_k]$. So, we have (i) and (ii).
 
By \eqref{eq:LF}, every derivation $\partial \in L$ equals $\sum_j f_j \partial_{x_j}$ for some $f_j \in \FF[x_1,\dots,x_{j-1}]$.
Since $\partial\in\Der(\cO(X))$, we have
\[
f_j = \partial(x_j) \in \cO(X) \cap \FF[x_1,\dots,x_{j-1}]
     \subset A_h[x_1,\dots,x_{j-1}].
   \]
     This gives us (iii) and (iv). 
\qedhere
\end{proof}

\begin{corollary}
   Let $L \subset \Der( \cO(X))$ be a Lie subalgebra contained in $\LND( \cO(X))$.
   Then $L$ is a solvable Lie algebra and a locally nilpotent subset of $\Der(\cO(X))$.
\end{corollary}
\begin{proof}
   The assertion follows from the same fact for $\jonq_u(k,A_h)$, Lemma~\ref{lm:L-structure}, and~\cite[Proposition]{Skutin}.
\end{proof}

\subsection{Nested unipotent groups}\label{sec:nested-unip}
Here we establish a correspondence between Lie subalgebras in $\LND(\cO(X))$ and nested unipotent subgroups of $\Aut(X)$.

\begin{proposition}\label{pr:exp-nested}
   Let $L \subset \Der( \cO(X))$ be a Lie subalgebra contained in $\LND(\cO(X))$.
   Then $U:=\exp(L)$ is a nested unipotent subgroup of $\Aut(X)$. 
   Conversely, if $U$ is a nested unipotent subgroup of $\Aut(X)$, then $\log(U)$ is a Lie subalgebra contained in $\LND(\cO(X))$.
\end{proposition}
To prove this, we will need the following lemma.
\begin{lemma}\label{lm:locnilp-findim}
   Let $S\subset \LND(\cO(X))$ be a finite locally nilpotent subset.
   Then the Lie subalgebra $W\subset \Der(\cO(X))$ generated by $S$ is finite-dimensional.
\end{lemma}
\begin{proof}
Let $S=\{d_1, \ldots, d_s\}$, and let $f \in \cO(X)$ be arbitrary. 
Then there exists $N \in \mathbb{N}$ such that any composition of $N$ basis elements applied to $f$ yields zero. 
Indeed, assume otherwise that no such $N$ exists. 
Then for some $i_1\le s$, no such $N$ exists for $d_{i_1}(f)$. 
Repeating this argument indefinitely produces an infinite sequence $d_{i_1},d_{i_2}, \ldots$ such that $d_{i_N}\cdots d_{i_1}\cdot f$ is non-zero for any $N \in \NN$. This contradicts the local nilpotency of $S$.

Therefore, $W \cdot f$ is finite-dimensional for any $f \in \cO(X)$. Applying this observation to a finite system of generators of the algebra $\cO(X)$, we conclude that $W$ is finite-dimensional itself.   
\end{proof}
 
\begin{proof}[Proof of Proposition~\ref{pr:exp-nested}]
     Let us choose any ascending chain of finite-dimensional subspaces $V_1\subset V_2\subset\ldots\subset L$ such that $L=\bigcup_i V_i$.
    By Lemma~\ref{lm:locnilp-findim}, for each $i$ the minimal subalgebra $W_i$ of $L$ containing $V_i$ is finite-dimensional. 
    Then $\exp(W_i)$ is an algebraic unipotent subgroup and $W_i$ is contained in $W_{i+1}$ for each $i$. Thus, $U=\bigcup_i\exp(W_i)$ is a nested unipotent subgroup. 

    For the converse, suppose $U=\bigcup_i U_i$, where $U_1\subset U_2\subset\ldots$ is an ascending sequence of unipotent algebraic subgroups, then $\log(U)$ is a directed union of Lie subalgebras $\log(U_i)\subset\LND(\cO(X))$, hence itself a Lie subalgebra contained in $\LND(\cO(X))$.
\end{proof}  

\subsection{Closedness}
Here we show that a nested unipotent subgroup of $\Aut(X)$ is closed in both $\Aut(X)$ and $\End(X)$.

\begin{lemma}\label{lm:end-closed}
   Consider an affine domain $R$ over $\K$ endowed with the ind-variety structure of Section~\ref{sec:ind-ring}. 
   Let $\mathcal{F}\subset R$ be a (possibly infinite) subset, and for each $f\in \mathcal{F}$ let $S_f\subset R$ be a closed subset. 
   Then the subset
 \[ S := \{\phi\mid  \phi(f)\in S_f \text{ for all } f\in \mathcal{F} \},\]   
   consisting of all $\K$-endomorphisms of $R$ sending each $f$ into $S_f$, is closed in $\End(R)$.
\end{lemma}
\begin{proof}
   For each $f\in \mathcal{F}$, consider the map \[\psi_f\colon \End(R)\to R,\, \phi\mapsto \phi(f).\]
   It is a morphism of ind-varieties. 
   Then the preimage $\psi_f^{-1}(S_f)$ is a closed subset of $\End(R)$,
   and $S$ is the intersection of closed subsets $\psi_f^{-1}(S_f)$ for $f\in\mathcal{F}$.
\end{proof}

\begin{corollary}\label{cr:jonq-closed}
   Let $R$ be an affine domain over $\K$.
   Then the following hold.
   \begin{enumerate}[(i)]
      \item The ind-subgroup $\Jonq_u(k,R)$ is closed in $\End(R[x_1,\ldots,x_k])$.
      \item 
   The map $\exp\colon \jonq_u(k,R)\to \Jonq_u(k,R)$ is an isomorphism of ind-subvarieties. 
   In particular, the nested structure of $\Jonq_u(k,R)$ is equivalent as an ind-structure to the induced one.
   \end{enumerate}
\end{corollary}
\begin{proof}
   (i) We apply Lemma~\ref{lm:end-closed} to the set $\mathcal{F}=\{x_1,\ldots,x_k\}$,
   letting \[S_{x_i}:= x_i+ R[x_1,\ldots,x_{i-1}].\]

   (ii)  Consider a filtration by finite-dimensional subspaces $R[x_1,\ldots,x_k]=\bigcup_i R_i$ and for each $d\in\NN$ denote 
   \[
   V_d := \bigoplus_{i=1}^k \left(R[x_1,\ldots,x_{i-1}]\cap R_d\right)\frac{\partial}{\partial x_i} \subset \jonq_u(k,R).
   \]
   This is a finite-dimensional subspace. Hence, by Lemma~\ref{lm:locnilp-findim}, the Lie subalgebra $W_d$ generated by $V_d$ is also finite-dimensional. 
   The image $\exp(W_d)$ contains the intersection of $\Jonq_u(k,R)$ with the $d$th filtration element of $\End(R[x_1,\ldots,x_k])$, see Section~\ref{sec:ind-ring}.
   So, the ind-structure of $\Jonq_u(k,R)$ given by the filtration $\bigcup_i \exp(W_i)$ is equivalent to the one induced from $\End(R[x_1,\ldots,x_k])$.
\end{proof}

\begin{corollary}\label{cr:local-closed}
   Let $R$ be an affine domain over $\K$ and $h$ be a non-zero element of $R$.
   Then the semigroup $\End(R)\cap\End(R_h)\subset\End(\Frac(R))$ is closed in both $\End(R)$ and $\End(R_h)$.
\end{corollary}
\begin{proof}
   By \cite[Theorem~4.1.1]{FK}, $S:= (R_h^\times \cup\{0\})\cap R$ is a closed subset of $R$. 
   Since the intersection is defined in $\End(R)$ as $\{\phi\mid \phi(h)\in S\}$, 
   and in $\End(R_h)$ as $\{\phi\mid \phi(f)\in R \text{ for all } f\in R\}$,
    the statement follows from Lemma~\ref{lm:end-closed}.
\end{proof}

\begin{proposition}\label{pr:unipotent-closed}
   A nested unipotent subgroup $U$ of $\Aut(X)$ is closed in both $\End(X)$ and $\Aut(X)$.
    Moreover, its nested structure is equivalent as an ind-structure to the induced one.
\end{proposition}
\begin{proof}
   By Lemma~\ref{lm:L-structure}, $L:=\log U$ is a Lie subalgebra 
   of $\jonq_u(k,\cO(X)_h)$ for some $k>0$ and $h\in\ker L$. 
   Then the following inclusions are closed:
   \begin{itemize}
      \item $L\subset \jonq_u(k,\cO(X)_h)$, since $L$ is a subspace;
      \item $U \subset \Jonq_u(k,\cO(X)_h)\subset \End(X_h)$ by Corollary~\ref{cr:jonq-closed};
      \item $U\subset \End(X)\cap \End(X_h)$;
      \item $U \subset \End(X)$ by Corollary~\ref{cr:local-closed};
      \item $U \subset \Aut(X)$, since $\Aut(X)$ is locally closed in $\End(X)$ by~\cite[Theorem~5.2.1]{FK}.
   \end{itemize}
   The second assertion follows from the equivalence of ind-structures on $\End(X)\cap\End(X_h)$ induced from $\End(X)$ and $\End(X_h)$, which can be verified directly.
\end{proof}

\begin{remark}\label{rm:Ga-closure} 
   For an arbitrary unipotent subgroup $U$, the subset $\log(U)$ need not be a Lie subalgebra.
   However, if $\log(U)$ is a Lie ring, i.e., closed under addition and the Lie bracket, then by~\cite[Theorem~1.5]{BPZ} it is a locally nilpotent subset.
   Indeed, the assumption 
   \[\dim_{\K(X)}(\K(X)\cdot\Lie U)<\infty\]
    of loc.\ cit. is satisfied, since $\dim_{\K(X)}(\K(X)\cdot\Der\cO(X))$ is finite.
   In this case, $\K\cdot\log(U)$ is a Lie subalgebra of $\LND(\cO(X))$ by~\cite[Corollary~1]{Skutin}, and its exponent is a nested unipotent subgroup equal to the closure of $U$ by Proposition~\ref{pr:exp-nested}. 
\end{remark}

\subsection{Connected nested subgroups}
Let $G$ be a connected nested subgroup in $\Aut(X)$.
   We may assume that $G=\bigcup G_i$, where $G_i\subset G_{i+1}$ are closed embeddings of connected algebraic subgroups of $\Aut(X)$, e.g., see the proof of~\cite[Proposition~1.6.2]{FK}. 
   In particular, each $G_i$ acts regularly on $X$.
   
   \begin{theorem}[{\cite[Section 2.2]{KPZ}}]\label{th:LU}
   There is a decomposition $G = L\ltimes R_u(G)$, where $L$ is a maximal reductive algebraic subgroup in $G$ and $R_u(G)$ is the unipotent radical of $G$. Moreover, one may choose a filtration $G=\bigcup G_i$ so that $G_i=L\ltimes R_u(G_i)$ and $R_u(G_i)=R_u(G)\cap G_i$.
   \end{theorem}
   
   In particular, $R_u(G)=\bigcup R_u(G_i)$ is a nested group filtered by algebraic unipotent subgroups. 
   
   \begin{theorem}\label{th:nested-closed}
      A connected nested subgroup $G\subset\Aut(X)$ is closed.
   \end{theorem}
   \begin{proof}
      Let us choose a filtration $G=\bigcup G_i$ as in Theorem~\ref{th:LU}.
   It is enough to prove that for any $d\in\NN$ there exists $k\in\NN$ such that
   \begin{equation}\label{eq:LU-findeg}
   G\cap \Aut(X)_{\le d}=G_k\cap \Aut(X)_{\le d}.
   \end{equation}
   Consider an element $g\in \Aut(X)_{\le d}$ and let $l\in L$, $u\in R_u(G)$ be such that $g=lu$. 
   Then $u\in L\cdot\Aut(X)_{\le d}$, which is an algebraic subset contained in $\Aut(X)_{\le j}$ for some $j$.
   Since $R_u(G)$ is closed in $\Aut(X)$ by Proposition~\ref{pr:unipotent-closed}(i), 
   and its nested ind-structure is equivalent to the induced one by Proposition~\ref{pr:unipotent-closed}(ii),
    there exists $k$ such that $R_u(G)\cap\Aut(X)_{\le j}=R_u(G_k)\cap\Aut(X)_{\le j}$. 
   So, $g=lu\in G_k$ and \eqref{eq:LU-findeg} holds.
   \end{proof}

\section{$\G_a$-Generated Unipotent Subgroups}\label{sec:non-nested}

In this section we discuss unipotent subgroups of $\Aut(X)$ generated by their $\G_a$-subgroups. 
We assume that the base field $\K$ is uncountable for the rest of the section.

      \begin{definition}\label{def:Ga-gen}
         We say that
         \begin{itemize}
            \item a subgroup $U\subset\Aut(X)$ is \emph{$\G_a$-generated} if it is generated by its $\G_a$-subgroups;
            \item a subset $M\subset \LND(\cO(X))$ is \emph{scaling-closed} if it is closed under scalar multiplication by $\K$. 
         \end{itemize}
      \end{definition}

      \begin{remark}\label{rm:Ga-gen}
         A subgroup of $\Aut(X)$ is $\G_a$-generated if and only if it is generated by an exponent of a scaling-closed subset of $\LND(\cO(X))$.   
      A unipotent subgroup of $\Aut(X)$ is $\G_a$-generated if and only if it is generated by its algebraic subgroups.
      \end{remark}

\begin{proposition}\label{pr:loc-nilp-logU}
   Let $M$ be a scaling-closed subset of $\LND(\cO(X))$ such that $\exp(M)$ generates a unipotent subgroup $U\subset\Aut(X)$.
   Then $M$ is locally nilpotent.
\end{proposition}

To prove this proposition, we need Lemmas~\ref{lm:logU-bracket} and~\ref{lm:Sk-1}.
We define the \emph{bracket closure} of a subset $L\subset\Der(\cO(X))$ to be the smallest subset of $\Der(\cO(X))$ that contains $L$ and is closed under the Lie bracket.
Note that $\log(U)$ itself is not necessarily closed under the Lie bracket.

\begin{lemma}\label{lm:logU-bracket}
   Let $L$ be the bracket closure of $M$.   
   Then any derivation in $L$ is locally nilpotent.
\end{lemma}
\begin{proof} 
   Given an arbitrary $\partial\in\LND(\cO(X))$, we extend it to a derivation of $\cO(X)[t]$ by letting $\partial(t)=0$.
    Then we have $t\partial\in\LND(\cO(X)[t])$ and $\exp(t\partial)\in\Aut(\cO(X)[t])$.

    An element $d\in L$ equals $F(a_1,a_2,\ldots,a_n)$ for some $a_1,a_2,\ldots,a_n\in M$ and some  expression $F$ in Lie brackets. Let $u_i=\exp(ta_i)\in\Aut(\cO(X)[t])$ for $i=1,\ldots,n$ and
    \begin{equation}\label{eq:c-iter-commutator}
      c=F(u_1,u_2,\ldots,u_n)\in \Aut(\cO(X)[t]),
    \end{equation}
    where $F$ is the same bracket expression, but brackets now denote taking commutators in the automorphism group, i.e., $[a,b]=aba^{-1}b^{-1}$ for $a,b\in\Aut(\cO(X)[t]).$ 
    We deduce by induction on $n$ that 
   \[
      c = 1+t^n d + h.o.t.
   \]
    where ``$h.o.t.$'' stands for higher order terms in $t$.
   Indeed, if $u_1,u_2\in\Aut(\cO(X)[t])$ are such that $u_i=1+t^{k_i}\partial_i+ h.o.t.$ 
   for $k_i>0$ and $\partial_i\in\Der(\cO(X))$, $i=1,2$, then we have
   \[
      [u_1,u_2] =1+ t^{k_1+k_2}[\partial_1,\partial_2]+h.o.t.
   \]

   If we substitute an element $\tau\in\K$ for $t$ in \eqref{eq:c-iter-commutator}, then $c$ becomes an element of $U$, which we denote by $c_\tau$. 
   In particular, $(c_\tau-\id)$ is a locally nilpotent operator on $\cO(X)$ for each $\tau\in\K$, see~\cite[Proposition~2.1.3]{Essen2000}. 
   
   Since $\K$ is uncountable,
    for any $f\in\cO(X)$ there exists $N=N(f)>0$ such that  $(c_\tau-\id)^N(f)=0$ for an infinite number of values $\tau$.
   Assume that $\cO(X)$ is generated by functions $f_1,\ldots,f_s$ and let 
   \[N=\max(N(f_1),\ldots,N(f_s)).\] 
   For any $i=1,\ldots,s$ the image $(c-\id)^N(f_i)\in\cO(X)[t]$ is a polynomial in $t$ that is equal to zero for an infinite number of values $\tau$ of $t$.
   Thus, the polynomial itself is zero, as well as its lower homogeneous component $(t^nd)^N(f_i)$. In other words, $d$ is an LND.
\end{proof}

\begin{remark}
The assumption that $\K$ is uncountable and $M$ is scaling-closed is used only in the proof of Lemma~\ref{lm:logU-bracket}. 
If one could establish Lemma~\ref{lm:logU-bracket} without these assumptions, it would follow that the closure of an arbitrary unipotent subgroup of $\Aut(X)$ is nested. 
In particular, a unipotent subgroup generated by an algebraic subset containing the identity would be algebraic, cf.~\cite{CRX}.
\end{remark}

The following lemma is a generalization of~\cite[Lemma 1]{Skutin}.
 
\begin{lemma}\label{lm:Sk-1}
   Let $T$ be a locally nilpotent %
   set of linear operators on a vector space $W$.
   Consider a subset $V\subset W$ such that $T(V)\subset V$
   and a subset $V'\subsetneq V$ that contains $0$.
Then we can find $v \in V \setminus V'$ such that $T(v)\subset V'$.
\end{lemma}
\begin{proof}
Assume the contrary and take some $v \in V \setminus V'$. Then there exist $A_1 \in T$ such that $A_1 v \notin V'$, $A_2 \in T$ such that $A_2 A_1 v \notin V'$ and so on.
Thus, we have an infinite sequence $A_1,A_2,\ldots\in T$ such that $A_k\cdots A_1 v\neq0$ for any $k>0$, which contradicts the local nilpotency condition.
\end{proof}

\begin{proof}[Proof of Proposition~\ref{pr:loc-nilp-logU}]
We proceed similarly to~\cite[Proposition]{Skutin}. 
Let
   $L$ be the bracket closure of  $M$, and 
   $S$ be a maximal locally nilpotent subset of $L$. 
Such $S$ exists due to~\cite[Corollary 2]{Skutin} and Zorn's Lemma. 
 Assuming that $S\neq L$, we arrive at a contradiction in three steps:
 \begin{itemize}
   \item By~\cite[Theorem 5]{Skutin}, $\text{ad}(S)$ is locally nilpotent on $\Der(\cO(X))$.
   \item Applying Lemma~\ref{lm:Sk-1} with 
   $T:=\text{ad}(S),$ $W:=\Der(\cO(X)),$ $V:=L,$ and $V':=S$,
   we can find $D \in L\setminus S$ such that $\text{ad}(S)(D)=[S,D] \subset S$.
   \item Since $D$ is locally nilpotent by Lemma~\ref{lm:logU-bracket}, we deduce from~\cite[Lemma 2]{Skutin} that the subset $S\cup\{D\}$ is locally nilpotent, a contradiction.
 \end{itemize}
 So, $M$ is contained in a locally nilpotent subset $S=L$, thus it is locally nilpotent itself. 
\end{proof}

In the following theorem we drop the solvability hypothesis in~\cite[Theorem~B]{KZ} for unipotent groups, since it holds automatically. 

\begin{theorem}
   \label{th:nested=closed}
   Assume that the base field $\K$ is uncountable and consider a unipotent subgroup $U$ of $\Aut(X)$. Then the following are equivalent:
   \begin{enumerate}[(i)]
      \item $U$ is $\G_a$-generated;
      \item $U$ is nested;
      \item $U$ is closed in $\Aut(X)$.
   \end{enumerate}
\end{theorem}

\begin{proof}
   Assume (i).
   Let $M$ be a scaling-closed subset of $\LND(\cO(X))$ such that $\exp(M)$ generates $U$.
   By Proposition~\ref{pr:loc-nilp-logU}, $M$ is locally nilpotent. 
   By~\cite[Corollary 1]{Skutin}, the Lie subalgebra $L$ generated by $M$ is contained in $\LND(\cO(X))$. 
   Then, by Proposition~\ref{pr:exp-nested}, $\exp(L)$ is a nested unipotent subgroup $U'$, which contains $\exp(M)$. 
   By~\cite[Theorem~C]{KZ}, the group $U'$ is solvable, hence $U\subset U'$ is also solvable.
   Now we have (ii) by~\cite[Theorem~B]{KZ}.

   Assume (ii). Then we have (iii) by Proposition~\ref{pr:unipotent-closed}.

   Assume (iii). Then for any $g\in U$ the corresponding $\G_a$-subgroup is also contained in $U$, and we have (i).
\end{proof}

\begin{corollary}\label{cor:unip-alggen}
   Let $U\subset \Aut(X)$ be a unipotent subgroup generated by a finite number of algebraic subgroups. Then $U$ is an algebraic subgroup.
\end{corollary}
\begin{proof}
   By Theorem~\ref{th:nested=closed}, $U$ is nested, so it admits a filtration $U=\bigcup_i U_i$ by algebraic subgroups. Since there are finitely many generating algebraic subgroups and the filtration is ascending, all of them are contained in a single $U_k$; hence $U=U_k$ is algebraic.
\end{proof}

\section{Maximal Unipotent Subgroups}\label{sec:max-unip}
In this section we describe maximal nested unipotent subgroups in terms of tuples of commuting LNDs and provide examples in small dimension.

\subsection{dJ-like subgroups}\label{sec:dJ-like}
Here we introduce, for an arbitrary affine variety $X$, analogues of the group $\Jonq_u(n,\K)\subset\Aut(\AA^n)$ using the following notion. 

\begin{definition}\label{def:width}
   We say that a subgroup $G\subset\Aut(X)$ is of \emph{width} $k$ if the dimension of the closure of a general $G$-orbit in $X$ equals $k$. 
   We say that a subgroup $H$ of a group $G\subset\Aut(X)$ is \emph{wide} (in $G$) if $H$ is of the same width as $G$.
\end{definition}

The following proposition summarizes the results of Sections~\ref{sec:ln-subsets} and~\ref{sec:nested-unip}.

\begin{proposition}\label{pr:U-wide-dJ}
   Let $U\subset\Aut(X)$ be a nested unipotent subgroup of width $k$.  Then there exists a principal affine $U$-invariant subset $X_h$ isomorphic to $\Spec(R)\times\AA^k$, where $R=\cO(X)^U_h,$ such that $U$ is a wide subgroup of $\Jonq_u(k,R)\cap\Aut(X).$
\end{proposition}
\begin{proof}
   By Lemma~\ref{lm:L-structure} applied to $L:=\log U$, there exist $h\in\cO(X)^U$ and $x_1,\ldots,x_k\in\cO(X)$ such that $\cO(X)_h=R[x_1,\ldots,x_k]$ for $R:=\cO(X)^U_h$, and $L\subset\jonq_u(k,R)$.
   Hence $X_h\cong\Spec(R)\times\AA^k$ and $U\subset\Jonq_u(k,R)\cap\Aut(X)$.
   Since both $U$ and $\Jonq_u(k,R)$ have width $k$, the subgroup $U$ is wide.
\end{proof}

Subgroups of the form $\Jonq_u(k,R)\cap\Aut(X)$ are in fact maximal among nested unipotent subgroups of the same width; see Definition~\ref{def:dJ-width} and Proposition~\ref{pr:dJ-like=Jdi}.

\begin{definition}\label{def:dJ-width}
   A nested unipotent subgroup $J\subset\Aut(X)$ is called \emph{de Jonqui\`eres-like} (or, for short, \emph{dJ-like}), if $J$ is maximal among nested unipotent subgroups of the same width as $J$.
\end{definition}

We describe dJ-like subgroups in terms of tuples of commuting derivations as follows.

\begin{definition}\label{def:gen-free-di}
   We say that a tuple of pairwise commuting locally nilpotent derivations $\partial_1,\ldots, \partial_k\in\LND(\cO(X))$, where $k\le\dim(X)$,
   is a \emph{generic frame} if the corresponding vector fields are linearly independent at the general point. 

   Given a generic frame $\partial_1,\ldots, \partial_k\in\LND(\cO(X))$, we further define:
   \begin{equation}\label{eq:D(di)}
      \mathcal{D}(\partial_1,\ldots,\partial_k):= \left(\bigoplus_{i=1}^k A_1^{-1}A_i\partial_i\right)\cap\Der(\cO(X)),  
   \end{equation}
   where $A_i=\bigcap_{j\ge i}\ker\partial_j\subset\cO(X)$ for $i=1,\ldots,k.$
   This is a Lie subalgebra of $\Der(\cO(X))$.
   We also denote:
   \[
      \J(\partial_1,\ldots,\partial_k):=\exp \mathcal{D}(\partial_1,\ldots,\partial_k).
   \]
   This is a unipotent subgroup of width $k$.
\end{definition}

\begin{proposition}\label{pr:dJ-like=Jdi}
   Let $k\le \dim(X)$ be a positive integer and let $L\subset\LND(\cO(X))$ be a subset. 
   Then the following conditions are equivalent:
   \begin{enumerate}[(i)]
      \item $L$ equals $\Lie J$ for a dJ-like subgroup $J\subset\Aut(X)$ of width $k$.
      \item $L$ equals $\mathcal{D}(\partial_1,\ldots,\partial_k)$ for a generic frame $\partial_1,\ldots,\partial_k$.
       \item There exist a subalgebra $A\subset\cO(X)$, a nonzero $h\in A$, and elements $x_1,\ldots,x_k\in\cO(X)$ satisfying conditions (i)--(ii) of Lemma~\ref{lm:L-structure}, such that
   \[L=\jonq_u(x_1,\ldots,x_k,A_h)\cap\LND(\cO(X)).\]
   Explicitly, $L$ is the set of derivations $\partial\in\LND(\cO(X))$ satisfying $\partial(A)=0$ and $\partial(x_i)\in A_h[x_1,\ldots,x_{i-1}]$ for $i=1,\ldots,k$.
   \end{enumerate}
\end{proposition}
\begin{proof}
   Assume (i).
   By applying Lemma~\ref{lm:L-structure} for $L=\Lie J$  and using the notation therein, we obtain
   \[
      \Lie J \subseteq \bigoplus_{i=1}^k A_h[x_1,\ldots,x_{i-1}]\frac{\partial}{\partial x_i}.
   \]
   Let us choose $d>0$ such that $\partial_i:=h^d \frac{\partial}{\partial x_i}$ belongs to $\Der(\cO(X))$ for all $i=1,\ldots,k$. Then $\mathcal{D}(\partial_1,\ldots,\partial_k)$ is a Lie algebra contained in $\LND(\cO(X))$ that contains $\Lie J$ and has the same kernel.
   Thus, $\J(\partial_1,\ldots,\partial_k)$ is a unipotent subgroup of $\Aut(X)$ that contains $J$ and is of the same width as $J$. 
   Since $J$ is maximal, we have (ii). 

   Assume (ii). 
   Since $\partial_1,\ldots,\partial_k$ are linearly independent at the general point, 
   there exists an element  $x_i\in\bigcap_{j\neq i}\ker\partial_j$ such that $x_i\notin\ker\partial_i$, where $i=1,\ldots,k$.

   Since $\partial_i$ and $\partial_j$ commute, $\partial_i(x_i)$ also belongs to $\bigcap_{j\neq i}\ker\partial_j$. 
   Applying $\partial_i$ to $x_i$ repeatedly, we may assume that $h_i:=\partial_i(x_i)\in A_1\setminus\{0\}$, i.e., that $x_i$ is a local slice of $\partial_i$.
   
   Now, $x_1,\ldots,x_k$ are algebraically independent, since the intersection of kernels of derivations is algebraically closed, see~\cite[Proposition 1.9.(d)]{Freudenburg}.
   By~\cite[Proposition 4.6]{DerksenLND}, we have 
   \[A_1^{-1}\cO(X)=\Frac(A_1)[x_1,\ldots,x_k].\]
   Moreover, $\partial_i$ equals $h_i\frac{\partial}{\partial x_i}$.
   Finally, taking $h$ as in the proof of Lemma~\ref{lm:L-structure} we obtain (iii).

   Assume (iii). Then $L$ is a Lie algebra and a locally nilpotent subset.
   It equals $\mathcal{D}(\partial_1,\ldots,\partial_k)$, where for each $i=1,\ldots,k$, we take $\partial_i=h^{d_i} \frac{\partial}{\partial x_i}$, with $d_i\in\NN$ chosen such that $\partial_i\in\Der(\cO(X))$.
   Assume that the nested unipotent group $U:=\exp(L)$ is not dJ-like.
   Then $U$ is strictly contained in some dJ-like subgroup $J$ of the same width as $U$.
   In particular, $\ker\Lie J$ equals $\ker L$.
   
   Consider an LND $\partial'\in\Lie J\setminus L$.
   Then $\partial'(x_j)\notin A_h[x_1,\ldots,x_{j-1}]$ for some $j\le k$, 
   and we assume that $j$ is a minimal index with such property. 
   Let us show that $L \cup\{\partial'\}$ is not a locally nilpotent set.

   For each $i=1,\ldots,k$ we consider a subset 
   \begin{equation}\label{eq:Si}
      S_i := (A_h\setminus\{0\})x_i + A_h[x_1,\ldots,x_{i-1}].
   \end{equation}
   By induction on the $x_i$-degree, applying elements of $L$ iteratively in the appropriate order sends any element of $A_h[x_1,\ldots,x_{i}]\setminus A_h[x_1,\ldots,x_{i-1}]$ to an element of $S_i$.
    On the other hand, one can send any element of $S_{i+1}$ to an element of $A_h[x_1,\ldots,x_{i}]\setminus A_h[x_1,\ldots,x_{i-1}]$ by applying $x_{i}^N\partial_{i+1}$ for some $N\gg0$.

    Let us also observe that $\partial'(y)$ does not belong to $A_h[x_1,\ldots,x_{j-1}]$ for any $y\in S_j$.
    Indeed, if we have $y=c x_j+P$, where $c\in A_h\setminus\{0\}$ and $P\in A_h[x_1,\ldots,x_{j-1}]$, then $\partial'(y)=c\partial'(x_j) + \partial'(P)$, where  $\partial'(P)$ belongs to $A_h[x_1,\ldots,x_{j-1}]$ and $\partial'(x_j)$ does not.
   By applying elements of $L$ iteratively in some order as indicated above, we can send $\partial'(y)$ back into some new $y\in S_j$.
   Thus, starting with $y:=x_j$ and repeating the procedure indefinitely, we
   conclude that $\Lie J$ is not a locally nilpotent subset, contradicting~\cite[Proposition]{Skutin}.
   We have (i).
\end{proof}

\begin{remark}\label{rm:local-slice-system}
   Given $\partial_1,\ldots,\partial_k$ as in Proposition~\ref{pr:dJ-like=Jdi}(ii), we may choose $x_1,\ldots,x_k$ as in Proposition~\ref{pr:dJ-like=Jdi}(iii) so that they comprise a \emph{system of local slices} of $\partial_1,\ldots,\partial_k$, i.e., $\partial_i(x_j)$ equals zero for $j\neq i$ and belongs to $A\setminus\{0\}$ for $j=i$. 
   The converse also holds, we may similarly choose $\partial_1,\ldots,\partial_k$ for given $x_1,\ldots,x_k$.
   See also~\cite[Proposition 3.27]{Freudenburg}.
\end{remark}

The following corollaries restate Proposition~\ref{pr:dJ-like=Jdi} in terms of unipotent subgroups.

\begin{corollary}\label{cor:dJ-like=Jdi}
   A dJ-like subgroup equals $\J(\partial_1,\ldots,\partial_k)$ for some generic frame $(\partial_1,\ldots,\partial_k)$. 
   Conversely, for any generic frame $(\partial_1,\ldots,\partial_k)$, the subset $\J(\partial_1,\ldots,\partial_k)\subset\Aut(X)$ is a dJ-like subgroup.
\end{corollary}
\begin{proof}
   This follows from Proposition~\ref{pr:dJ-like=Jdi}(i--ii).
\end{proof}

\begin{corollary}\label{cor:dJ-like-flag}
   Consider a dJ-like subgroup $J\subset\Aut(X)$ of width $k$. 
   There exists a coflag of $\AA^1$-fibrations of affine schemes
   \[
      X\stackrel{\pi_k}{\to} X_k\stackrel{\pi_{k-1}}{\to}\ldots\stackrel{\pi_1}{\to} X_1
   \]
   such that $J$ consists of all unipotent elements of $\Aut(X)$ that stabilize the coflag and act as the identity on $X_1$.
\end{corollary}
\begin{proof}
   By Proposition~\ref{pr:dJ-like=Jdi}, $J$ equals $\J(\partial_1,\ldots,\partial_k)$ for some generic frame $\partial_1,\ldots,\partial_k$. 
   Denote $A_i=\bigcap_{j\ge i}\ker \partial_j$ as in Definition~\ref{def:gen-free-di}
   and let $X_i:=\Spec A_i$ and $X_{k+1}:=X$.
   Then the natural map $\pi_i\colon X_{i+1}\to X_i$ is an $\AA^1$-fibration, since its general fiber is a $\G_a$-orbit of $\exp(\partial_i)$, for $i=1,\ldots,k$.

   By construction, elements of $J$ stabilize the coflag and are trivial on $X_1$.
    It therefore remains to show the converse: any unipotent $u\in\Aut(X)$ that stabilizes the coflag and is trivial on $X_1$ must belong to $J$.
   Let $\partial=\log(u)$. Then $A_1\subset \ker \partial$, and $A_i$ is stabilized by $\partial$ for all $i=1,\ldots,k$.

   Let $x_1,\ldots,x_k\in\cO(X)$ be a system of local slices of $\partial_1,\ldots,\partial_k$ as in Remark~\ref{rm:local-slice-system}.
   Let us show that $\partial(x_i)\in A_h[x_1,\ldots,x_{i-1}]$ for all $i=1,\ldots,k$.
   Since $\partial$ is an LND, there exists $P\in A_h[x_1,\ldots,x_i]\setminus A_h[x_1,\ldots,x_{i-1}]$ such that $\partial(P)\in A_h[x_1,\ldots,x_{i-1}]$.
   If $\deg_{x_i}(\partial(x_i))\ge2$, then $\deg_{x_i}(\partial(P))>\deg_{x_i}(P)$, a contradiction.
   So, we have $\partial(x_i)=c x_i + Q$ for some $c,Q\in A_h[x_1,\ldots,x_{i-1}]$.
   In particular, there exists $P=c'x_i+Q'\in S_i$, where $S_i$ as in \eqref{eq:Si}, such that $\partial(P)\in A_h[x_1,\ldots,x_{i-1}]$.
   Since  
   \[\partial(P)=(\partial(c')+cc')x_i+ c'Q+\partial(Q'),\]
we have $\partial(c')=-cc'\in (c')$. By~\cite[Principle~5]{Freudenburg}, we see that $\partial(c')=0$, hence $c=0$.
   The assertion follows.
\end{proof}

\begin{corollary}\label{cor:Ji-in-Jk}
   A dJ-like subgroup $\J(\partial_1,\ldots,\partial_k)$ contains the subgroup $\J(\partial_i,\ldots,\partial_k)$ for any $i\le k$.
\end{corollary}
\begin{proof}
   Let $X\to X_k\to\ldots\to X_1$ be the coflag as in Corollary~\ref{cor:dJ-like-flag}.
   Then $\J(\partial_i,\ldots,\partial_k)$ consists of unipotent elements that stabilize this flag and are trivial on $X_i$.
   Thus, it is contained in $\J(\partial_1,\ldots,\partial_k)$.
\end{proof}

In fact, $\J(\partial_1,\ldots,\partial_k)$ does not contain any other dJ-like subgroups, see Proposition~\ref{pr:J'<J}.
In order to prove it, we need the following definition.

   \begin{definition}\label{def:d-dJ-order}
      Let $J:=\J(\partial_1,\ldots,\partial_k)\subset \Aut(X)$ be a dJ-like subgroup.
      We define the \emph{order} of a nonzero LND $\partial\in\Lie J$ in $J$ to be the minimal $i$ such that $\partial\in\mathcal{D}(\partial_{k-i+1},\ldots,\partial_k)$.
      We denote by  $\ord_J\partial$  the order of $\partial$ in $J$.

      We say that an ordered tuple $(\partial_1',\ldots,\partial_s')$ of elements of $\Lie J$ is \emph{triangular} in $J$ if their orders in $J$ are $s,s-1,\ldots,1$ respectively.
   \end{definition}

    Definition~\ref{def:d-dJ-order} does not depend on the choice of $\partial_1,\ldots,\partial_k$ due to the following lemma. 

   \begin{lemma}\label{lm:d-ord}
      Consider an LND $\partial\in\Lie J$ of order $i$, where $J:=\J(\partial_1,\ldots,\partial_k)\subset \Aut(X)$ is a dJ-like subgroup.
      Then we have
      \[
         A_\partial:=\{f\in \cO(X)\mid f\partial\in \Lie J\}=\bigcap_{j\ge i}\ker\partial_j.
      \] 
      In particular, the subset $A_\partial$ is a subalgebra of $\cO(X)$ of transcendence degree $\dim X-i$.
   \end{lemma}
   \begin{proof}
      In the notation of \eqref{eq:D(di)},
       we have $\partial=\sum_{j=l}^k f_j\partial_j$ for some $f_j\in A_1^{-1}A_j$, where $l=k-i+1$ and $f_{l}\neq 0$.
      A regular function $f\in\cO(X)$ belongs to $A_\partial$ if and only if $ff_j\in A_1^{-1}A_j$ for all $j\ge l$.
      This is equivalent to the condition $ff_{l}\in A_1^{-1}A_{l}$.
      Then we have $A_\partial=A_l$ and the assertion follows.
   \end{proof}
   
\begin{proposition}\label{pr:J'<J}
   Consider dJ-like subgroups $J,J'\subset\Aut(X)$ of widths $k$ and $k'$.
   Write $J'=\J(\partial_1',\ldots,\partial_{k'}')$ for some generic frame $(\partial_1',\ldots,\partial_{k'}')$, which exists by Corollary~\ref{cor:dJ-like=Jdi}.
   
   Then we have $J'\subseteq J$ if and only if $(\partial_1',\ldots,\partial_{k'}')$ is a triangular tuple in $J$.
   Moreover, if $J'\subseteq J$, then we have $J'=\J(\partial_l,\ldots,\partial_k)$, where $l=k-k'+1$.
\end{proposition}
\begin{proof}
   Let $J=\J(\partial_1,\ldots,\partial_k)$.
   If  $(\partial_1',\ldots,\partial_{k'}')$ is triangular in $J$, then
   we have 
   \[
      \bigcap_{i=k-j}^k\ker\partial_i=\bigcap_{i=k'-j}^{k'}\ker\partial_i'
   \]
   for any $j=0,1,\ldots,k'-1$. Thus, $J'$ is contained in $\J(\partial_{k-k'+1},\ldots,\partial_k)$ by \eqref{eq:D(di)} and in $J$ by Corollary~\ref{cor:Ji-in-Jk}.

   Now assume that $J'$ is contained in $J$. 
   By Lemma~\ref{lm:d-ord}, $\ord_J\partial_i'\le \ord_{J'}\partial_i'$ for any $i=1,\ldots,k'$.
   If the inequality is strict for some $i$, then $(\partial_i',\ldots,\partial_{k'}')$ is not a generic frame, a contradiction.
   Thus, $(\partial_1',\ldots,\partial_{k'}')$ is triangular in $J$.

   So, $J'$ is a dJ-like subgroup of width $k'$ contained in a dJ-like subgroup $\J(\partial_{l},\ldots,\partial_k)$ of the same width $k'$, where $l=k-k'+1$.
   Thus, they coincide. 
\end{proof}

\subsection{Maximal unipotent subgroups: examples} \label{sec:max-ex}
Here we provide examples of maximal unipotent subgroups of different widths, including the case of a variety $X$ of dimension 2 or 3.

\begin{remark}\label{rm:k=n}
     Consider a dJ-like subgroup $J\subset\Aut(X)$ of width $n=\dim X$. 
     In this case, letting $J=\J(\partial_1,\ldots,\partial_n)$, $J$ contains the commutative algebraic unipotent subgroup $U=\exp(\langle \partial_1,\ldots,\partial_n\rangle_\K)$, which acts with an open orbit on $X$. 
     Then $U$ acts transitively on $X$ and we have $X\cong\AA^n$, e.g., see~\cite[Corollary~2.2]{GMM17}. 
     Thus, we have $J=\Jonq_u(n,\K)$ in coordinates $x_1,\ldots,x_n$ as in Proposition~\ref{pr:dJ-like=Jdi}, recovering Theorem~\ref{th:KZ-intro}.
   \end{remark}

   \begin{example}
      Let $X=\AA^2$ and consider a dJ-like subgroup $J=\J(\partial_1)\subset\Aut(\AA^2)$ for some $\partial_1\in\LND(\cO(\AA^2))$. By Rentschler's Theorem~\cite{Ren}, see also~\cite[Theorem 4.1]{Freudenburg}, after a suitable change of coordinates on $\AA^2$, we have $\partial_1=f(x)\frac{\partial}{\partial y}$ for some $f\in\K[x]$.
      Since $\partial_1$ and $\frac{\partial}{\partial y}$ have the same kernel, they define the same dJ-like subgroup, so $J=\J(\frac{\partial}{\partial y})$.
      Thus, any maximal nested unipotent subgroup of $\Aut(\AA^2)$ is conjugate to $\Jonq_u(2,\K)$.
   \end{example}     
        
\begin{example}
   Let $X$ be an affine surface not isomorphic to $\AA^2$.
   By Remark~\ref{rm:k=n}, every  dJ-like subgroup of $\Aut(X)$ is of width one, hence of the form $\J(\partial)$, where $\partial\in\LND(\cO(X))$. 
   In particular, all nested unipotent subgroups of $\Aut(X)$ are commutative.
\end{example}     

\begin{example}
   By~\cite[Proposition 5.40]{Freudenburg}, if two non-equivalent LNDs $D,E$ on $\K[x,y,z]$ commute, then up to an automorphism of $\K[x,y,z]$ we have $(\ker D) \cap (\ker E)=\K[x]$. In particular, $D$ and $E$ are of rank at most two, see the definition in~\cite[Section~3.2.1]{Freudenburg}. Indeed, as mentioned in loc.\ cit. after Proposition~5.40, this proposition indicates that a rank-three $\G_a$-action on $\AA^3$ cannot be extended to a $\G_a^2$-action, see also~\cite{DF}. 
   
   Thus, for any LND $\partial$ of rank three the dJ-like subgroup $\J(\partial)$ is maximal and of width one. For examples of rank-three LNDs of $\K[x,y,z]$, see~\cite{Freu98},~\cite[Chapter~5]{Freudenburg}, and~\cite{DasGa24}.
   In general, all commutative maximal dJ-like subgroups of $\Aut(\AA^3)$ can be described using the criterion from~\cite[Proposition~4.1]{DasGa24}.
\end{example}

\begin{example} 
   Let $X=\AA^3=\Spec\K[x,y,z]$ and consider the Nagata automorphism $\nu=\exp(f\partial)$, where $f=(xz-y^2)$ and $\partial=x \frac{\partial}{\partial y} + 2y \frac{\partial}{\partial z}$, see, e.g.,~\cite[Section 3.8.1]{Freudenburg}. 
   
   Since $\partial$ commutes with $\partial_z:=\frac{\partial}{\partial z}$, we have $\nu\in\J(\partial_z,\partial)$.
   If $\J(\partial_z,\partial)$ is not maximal, then
   there exists a dJ-like subgroup $\J'$ of width three that contains $\J(\partial_z,\partial)$. 
   On the other hand, $\J'$ is conjugate to $\Jonq_u(3)$ by Remark~\ref{rm:k=n}, so $\nu\in\Jonq(3,\K)$ up to conjugation, contradicting~\cite{Bass}.
      Thus, $\J(\partial_z,\partial)$ is a maximal nested unipotent subgroup of $\Aut(\AA^3)$ of width two. %

   The system of slices for $\partial_z, \partial$ is given by $\frac{f}{x},\frac{y}{x}$, so the corresponding localization is $\AA^3_x=\{x\neq0\}\subset \AA^3$. 
   Its direct product structure $\AA^3_x\cong \AA^1\setminus\{0\}\times\AA^2$ is given by the decomposition
   \[
      \cO(\AA^3_x)=\K[x,x^{-1},y,z]=\K[x,x^{-1}]\otimes\K\left[\frac{f}{x},\frac{y}{x}\right].
   \]
   So, we have
   \[
    \mathcal{D}(\partial_z,\partial)\subset
    \K[x,x^{-1}]\frac{\partial}{\partial z}
    \oplus
    \K[x,x^{-1},f]\partial.
   \]
\end{example}

\begin{remark}
   If $X=\AA^n$ and $\J=\J(\partial_1,\ldots,\partial_{n-1})$ with $\partial_{n-1}$ of rank at least two, then $\J$ is a maximal nested unipotent subgroup. Indeed, otherwise, up to conjugation, $\J$ is contained in $\Jonq_u(n,\K)$, and by Proposition~\ref{pr:J'<J}, $\partial_{n-1}$ is equivalent to $\frac{\partial}{\partial x_n}$.
\end{remark}

\begin{question}
   How can one describe dJ-like subgroups of $\Aut(X)$ of width $\dim X-1$?
\end{question}

\appendix

\section{Commutative Unipotent Subgroups}\label{sec:max-commut}
 In this appendix we describe connections between commutative unipotent subgroups, generic frames of derivations, and dJ-like subgroups.
We also discuss when a dJ-like subgroup is maximal.

\subsection{Maximal commutative unipotent subgroups}
   In~\cite[Theorem~A and Proposition~4.1]{RvS}, the authors prove that any maximal commutative unipotent subgroup of $\Aut(X)$ has the form $\mathrm{R}_X(U)$ for some algebraic commutative unipotent subgroup $U\subset\Aut(X)$, where 
   \[
      \mathrm{R}_X(U) = \langle \exp(f\partial)\in\Bir(X)\mid f\in\K(X)^U, \partial\in\Lie U\rangle\cap\Aut(X),
   \]
   where $\Bir(X)$ is the group of birational automorphisms of $X$.
      More generally, for an arbitrary commutative subset $S\subset\LND(\cO(X))$, we define $\mathrm{R}_X(S)$ by its Lie algebra: letting $K:=\ker_{\K(X)}S$,
      \begin{equation}\label{eq:RxU-def}
      \mathrm{R}_X(S):=\exp\left(\langle S\rangle_{K}\cap\Der(\cO(X))\right).     
      \end{equation}
   Then $\mathrm{R}_X(U)=\mathrm{R}_X(\log U)$.

\begin{lemma}\label{lm:comm-width}
   Let $U \subset \Aut(X)$ be a commutative unipotent subgroup. Then %
   \begin{enumerate}[(i)]
         \item the width of $U$ equals the dimension of the $\K(X)^U$-space $\langle\Lie U\rangle_{\K(X)^U}$;
      \item $U$ is a wide subgroup of $\mathrm{R}_X(U)$;
      \item we have $\mathrm{R}_X(U)=\mathrm{R}_X(V)$ for any wide subgroup $V \subset U$.
   \end{enumerate}
\end{lemma}
\begin{proof}
   The first assertion is immediate from the definition of width. Then the second one follows from \eqref{eq:RxU-def}.
   
   Let us show (iii). Since $V\subset U$ is wide, the invariant fields $\K(X)^U$ and $\K(X)^V$ coincide.
   Applying (i) to $U$ and $V$, we see that the $\K(X)^U$-subspaces spanned by $\Lie U$ and $\Lie V$ coincide.
\end{proof}

   \begin{proposition}\label{pr:max-commut}
      Let $U\subset \Aut(X)$ be a commutative unipotent subgroup of width $k$. 
      Then the following are equivalent.
      \begin{enumerate}[(i)]
         \item $U$ is maximal among commutative subgroups;
         \item $U$ equals 
         $\mathrm{R}_X(\{\partial_1,    \ldots,\partial_k\})$
         for a generic frame $\partial_1,\ldots,\partial_k\in\LND(\cO(X))$;
         \item there exist a $U$-invariant $h\in \cO(X)^U$ and an isomorphism $X_h \cong Y \times \AA^k$, where $Y=\Spec \cO(X)_h^U$ is affine, such that $U$ consists of all automorphisms of $X$ acting by translations on $\AA^k$-fibers of $X_h$.
      \end{enumerate}
   \end{proposition}
   \begin{proof}
      By assumption, $\log U$ contains $k$ LNDs, say $\partial_1,\ldots,\partial_k$, that are linearly independent as vector fields at a general point. Since $U$ is commutative, they form a generic frame.
      We denote $L = \langle\partial_1,    \ldots,\partial_k\rangle_\K$.

      By Lemma~\ref{lm:comm-width}, $\log U$ is contained in $L \otimes_\K \K(X)^U$. This gives the implication (i)~$\Rightarrow$~(ii).

      Let us take $h$ as in Lemma~\ref{lm:L-structure} applied to $L$. 
      Consider the map $X_h\to Y$, where $Y = \Spec(\cO(X)_h^U)$. 
      By Lemma~\ref{lm:L-structure}, $\cO(X)_h\cong\cO(X)_h^U[x_1,\ldots,x_k]$, so the fibers of $X_h\to Y$ are isomorphic to $\AA^k$. By~\eqref{eq:RxU-def}, automorphisms of $X$ acting by translations on these fibers are exactly elements of $\mathrm{R}_X(\{\partial_1,\ldots,\partial_k\})$, giving the implication (ii)~$\Rightarrow$~(iii).

      Finally, assume (iii) and consider an automorphism $g\in\Aut(X)$ that commutes with $U$. 
      By the choice of $h$ the subgroup $V:=\exp(L)$ acts transitively on the $\AA^k$-fibers of $X_h$.
      Since $g$ commutes with the subgroup
      $\exp(L\otimes_\K \cO(X)^U)\subset U$, we have $g(U x)=U (g x)$ for any $x\in X_h$.
      In particular, $g|_{U x}$ is an automorphism of $\AA^k$ that commutes with the subgroup of translations.
      Hence, $g$ is a translation on any $\AA^k$-fiber of $X_h$, so we have (i).
   \end{proof}

   \begin{remark}
       Let $U\subset \Aut(X)$ be a commutative unipotent subgroup  and $h$ be as in Lemma~\ref{lm:L-structure}. In the notation of~\cite[Section 10.4]{FK}, we have
      \[
         \mathrm{R}_X(U) = U(\cO(X)_h^U) \cap \Aut(X).
      \]
      In particular, the equivalence of (ii) and (iii) in Proposition~\ref{pr:max-commut} is a variant of~\cite[Proposition~10.4.12]{FK}.
   \end{remark}

\subsection{Maximal commutative subgroups in dJ-like subgroups}

\begin{proposition}\label{pr:R<J}
   Consider a maximal commutative unipotent subgroup $R\subset\Aut(X)$ and a dJ-like subgroup $J\subset\Aut(X)$.
   Then we have $R\subset J$ if and only if there exists
   a generic frame $(\partial_1',\ldots,\partial_s')$ triangular in $J$ such that $R=\mathrm{R}_X(\{\partial_1',\ldots,\partial_s'\})$.
\end{proposition}

The proof of Proposition~\ref{pr:R<J} relies on the following lemma on commuting derivations in the Lie algebra of a dJ-like subgroup.
\begin{lemma}\label{lm:commuting-pair}
   Consider a generic frame $(\partial_1,\ldots,\partial_k)$ and a pair of derivations $\partial,\partial'\in\mathcal{D}(\partial_1,\ldots,\partial_k)$. 
   Let $\partial=f_1\partial_1+\ldots+f_k\partial_k$ and $\partial'=f_1'\partial_1+\ldots+f_k'\partial_k$. 
   Then $\partial$ and $\partial'$ commute if and only if we have $\partial(f_j') = \partial'(f_j)$ for all $j=1,\ldots,k$.

   Moreover, assume that $[\partial,\partial']=0$ and $f_1'=\ldots=f_{l-1}'=0$ for some $l\le k$. Then $\partial(f_l')=0$.
\end{lemma}
\begin{proof}
  We have
   \[
      [\partial,\partial']=\left[\sum_{j=1}^k f_j\partial_j,\sum_{j=1}^k f_j'\partial_j\right]=\sum_{j=1}^k \left(
         \partial(f_j') - \partial'(f_j)
       \right)\partial_j.
   \]
   So, the first assertion follows.
   If we have $f_1'=\ldots=f_{l-1}'=0$, then $\partial'(f_l)$ is zero due to $f_l\in\bigcap_{j=l}^k \ker_{\K(X)}\partial_j$.
   The second assertion follows from the first one.
\end{proof}

\begin{proof}[Proof of Proposition~\ref{pr:R<J}]
   Assume that $R$ is contained in $J:=\J(\partial_1,\ldots,\partial_k)$ and let $s$ be the width of $R$.
   By Proposition~\ref{pr:max-commut}(ii), there exists a generic frame $(\partial_1',\ldots,\partial_s')$ such that $R=\mathrm{R}_X(\{\partial_1',\ldots,\partial_s'\})$.
   Then we have $\ord_J\partial_j'\ge k-s+1$ for each $j=1,\ldots,s$, that is, $R\subset \J(\partial_{k-s+1},\ldots,\partial_k)$. 
    Thus, we may assume that $k=s$.
   
   Let $\partial_i'=\sum_{j=1}^k f_{i,j}\partial_j$ for $i=1,\ldots,k$, where $f_{i,j}$ are as in \eqref{eq:D(di)}.
   We will transform the matrix $(f_{i,j})$ into the row echelon form by linear transformations of $\partial_1',\ldots,\partial_k'$ over the invariant field $\K(X)^J$.
   We can perform the first step of the Gaussian elimination algorithm, since $f_{i,1}$ belongs to $\K(X)^J$ for $i=1,\ldots,k$, and at least one of them is non-zero.
   
   Assume that we have performed $d-1$ steps of the Gaussian elimination over $\K(X)^J$. Namely, we have $f_{i,j}=0$ for $j<\min(i,d)$.
   We claim that
   \begin{equation}\label{eq:Gauss}
      \begin{cases}
         f_{i,i}\in \K(X)^J\setminus\{0\}&\text{ for }i<d, \\
         f_{i,d}\in \K(X)^J &\text{ for } i\ge d.
      \end{cases}
   \end{equation}
      Indeed, since $\partial_1',\ldots,\partial_k'$ form a generic frame, we have $f_{i,i}\neq0$ for $i<d$.
   By Lemma~\ref{lm:commuting-pair}, functions $f_{i,i}$ for $i<d$ and $f_{i,d}$ for $i\ge d$ are annihilated by $\partial_1',\ldots,\partial_k'$.
   Since exponents $\exp\partial_1',\ldots, \exp\partial_k'$ generate a wide subgroup of $J$, we have \[\bigcap_{i=1}^k \ker_{\K(X)}\partial_i'=\K(X)^J.\]
   Thus, we obtained \eqref{eq:Gauss} and we can proceed with the $d$th step of the Gaussian elimination.
   In the end, we arrive at a tuple that is triangular in $J$ and generates $R$, as required.

   The converse implication is immediate.
\end{proof}

\subsection{Maximality criterion for a dJ-like subgroup}

   \begin{proposition}\label{pr:dJ-max-criterion}
      A dJ-like subgroup $J\subset\Aut(X)$ of width $k$ fails to be a maximal nested unipotent subgroup if and only if 
      there exist commutative unipotent subgroups $G, H \subset \Aut(X)$ such that $H$ is a wide subgroup of $J$ and a non-wide subgroup of $G$.
   \end{proposition}
\begin{proof}
For the forward implication, assume that $J$ is not maximal.
Then there exists a dJ-like subgroup of width $k+1$ that contains $J$, say, $J':=\J(\partial_0,\ldots,\partial_k)$.
Thus, we have $J=\J(\partial_1,\ldots,\partial_k)$, and we may take $H:=\exp(\langle \partial_1,\ldots,\partial_k\rangle_\K)$ and $G:=\exp(\langle \partial_0,\ldots,\partial_k\rangle_\K)$.

Assume now that there exist $G,H$ as in the statement.
   We can find a generic frame $(\partial_1',\ldots,\partial_k')$ in $\log H$ and some $\partial_0'\in\log G$ that is linearly independent with $\partial_1',\ldots,\partial_k'.$
   By Proposition~\ref{pr:R<J}, there exists a generic frame $(\partial_1,\ldots,\partial_k)$ triangular in $J$ such that $R:=\mathrm{R}_X(\{\partial_1,\ldots,\partial_k\})$ equals $\mathrm{R}_X(\{\partial_1',\ldots,\partial_k'\})$.
   In particular, $J$ equals $\J(\partial_1,\ldots,\partial_k)$ by Proposition~\ref{pr:J'<J}.

      Since $\partial_0'$ commutes with $\partial_i'$, $i=1,\ldots,k$, it preserves $\K(X)^R$, and we have \[[\partial_0',\log R]\subseteq \log R.\]
   
   Let $\partial_0:=\partial_0'-\sum_{i=1}^k[\partial_0',\partial_i]$.
   Then $\partial_0$ belongs to $\J(\partial_0',\ldots,\partial_k')$ and is an LND that commutes with $\partial_1,\ldots,\partial_k$.
   We conclude that $J\subsetneq\J(\partial_0,\ldots,\partial_k)$, so $J$ is not maximal.
\end{proof}

\end{document}